\documentclass[oneside,reqno]{amsart}

      \usepackage{amssymb}
      
      \usepackage{amsmath}
      \usepackage[noend]{algpseudocode}
      \usepackage{footmisc}
      
      \usepackage{easybmat}
      
      \usepackage{epsfig,graphicx,multirow,pdfpages}
      
      \usepackage{adjustbox}
      
      \usepackage[boxed,noend]{algorithm2e}
      
      \usepackage{hyperref}

      \newenvironment{bvec}[1]{\begin{BMAT}(r)[1pt]{c}{#1}}{\end{BMAT}}
      \newenvironment{pbvec}[1]{\left(\!\begin{bvec}{#1}}{\end{bvec}\!\right)}
      
      \usepackage{enumitem}
      \usepackage{xcolor}
      
      \theoremstyle{plain}
      \newtheorem{theorem}{Theorem}[section]

      \theoremstyle{definition}

      \theoremstyle{remark}

      \makeatletter
      \def\@setcopyright{}
      \def\serieslogo@{}
      \makeatother

\makeatletter
\def\BState{\State\hskip-\ALG@thistlm}
\makeatother

\makeatletter

\def\mydashbox#1#2{%
\setbox0\hbox{#2}%
\dimen0\ht0
\advance\dimen0\dp0
\setbox2\vbox to \dimen0{{\color{#1}\leaders\vbox{\vskip2pt\hrule height 2pt width .3pt}\vfill}}%
\ht2=\ht0
\dp2=\dp0
\box2
\unhbox0
}

        \renewcommand{\algocf@Vsline}[1]{%
                \strut\par\nointerlineskip%
                \algocf@push{\skiprule}%
                \hbox{%
            \mydashbox{black}{%
                        \vtop{\algocf@push{\skiptext}%
                        \vtop{\algocf@addskiptotal\advance\hsize by -\skiplength #1}}%
                }%
            }
                \algocf@pop{\skiprule}%
        }
\makeatother

\begin{document}

   \author{$\;\;\,$Greg Hurst\,*}
   \address{WOLFRAM RESEARCH INC., 100 TRADE CENTER DRIVE, CHAMPAIGN, IL 61820, U.S.A.}
   \email{ghurst@wolfram.com}
   \footnotetext{*Wolfram Research Inc., 100 Trade Center Drive, Champaign, IL 61820, U.S.A.}
   \footnotetext{\emph{E-mail address}: \href{mailto:ghurst@wolfram.com}{\texttt{ghurst@wolfram.com}}}
   
   % title

   \title[Computing the Mertens Function]{Computations of the Mertens Function and Improved Bounds on the Mertens Conjecture}
   %\title[Computing the Mertens Function]{Computations Regarding the Mertens Function}
   
   \begin{abstract}
     The Mertens function is defined as $M(x) = \sum_{n \leq x} \mu(n)$, where $\mu(n)$ is the M\"obius function. 
     The Mertens conjecture states $|M(x)/\sqrt{x}| < 1$ for $x > 1$,
     which was proven false in 1985 by showing $\liminf M(x)/\sqrt{x} < -1.009$ and $\limsup M(x)/\sqrt{x} > 1.06$.
     The same techniques used were revisited here with present day hardware and algorithms, giving improved lower and upper bounds of $-1.837625$ and $1.826054$.
     In addition, $M(x)$ was computed for all $x \leq 10^{16}$, recording all extrema, all zeros, and $10^8$ values sampled at a regular interval. 
     Lastly, an algorithm to compute $M(x)$ in $O(x^{2/3+\varepsilon})$ time was used on all powers of two up to $2^{73}$.
   \end{abstract}

   %\date{\today}

% This ends the top matter information.
% We can now tell LaTeX to display the top matter.

   \maketitle

% Having displayed the top matter, we now proceed to the body of the
% article.

% The body of the article is divided into sections.
% Each section begins with a \section command.

   \section{Introduction}
   \label{intro}
   
   The M\"obius function $\mu(n)$ is an arithmetic function defined by
   
   $$ \mu(n) = \begin{cases}(-1)^{\omega(n)} & \mbox{ if } n \mbox{ is a square-free integer}\\ 0&\mbox{ otherwise },\end{cases} $$
   where $\omega(n)$ is the number of prime factors of $n$. The Mertens function is the summatory function of the M\"obius function, i.e. $$ M(x) = \sum_{n \leq x} \mu(n). $$
   This is a well known function in number theory, appearing in many identities. Its Mellin transform gives 
   $$ \frac{1}{\zeta(s)} = s \int_1^\infty M(x)x^{-s-1}dx \;\;\text{ for }\;\; \text{Re}(s) > 1, $$
   where $\zeta(s)$ is the Riemann zeta function.
   %which means that the growth of $M(x)$ has \emph{direct} consequences towards the Riemann hypothesis. 
   If $M(x) = O(x^{1/2+\varepsilon})$, the integral would converge for $\text{Re}(s) > 1/2$, implying that $1/\zeta(s)$ has no poles in this region and that the Riemann hypothesis is true.
   Conversely, if $M(x) = \nolinebreak \Omega(x^\alpha)$ for some $\alpha > 1/2$, then the Riemann hypothesis is false.
   
   Defining $q(x) = M(x)/\sqrt{x}$, a long standing conjecture of Mertens stated $|q(x)| < 1$ for $x > 1$.
   In 1985 this was shown to be false by Odlyzko and te Riele who showed $\liminf q(x) < -1.009$ and $\limsup q(x) > 1.06$ \cite{OR}.
   However no explicit counterexample was found. Since then Best and Trudgian have improved these bounds to $\liminf q(x) < -1.6383$ and $\limsup q(x) > 1.6383$ \cite{BT}.
   This paper describes techniques similar to those of Odlyzko and te Riele and establishes $\liminf q(x) < -1.837625$ and $\limsup q(x) > 1.826054$.
   
   To better understand $M(x)$ and $q(x)$, some have computed $M(x)$ at every integer up to a given bound.
   The most recent and extensive results are due to Kotnik and van de Lune, who computed $M(x)$ for all $x \leq 10^{14}$ \cite{KL}.
   In this paper, these results are extended by computing $M(x)$ for all $x \leq 10^{16}$. For $x$ in this range
   \newlength{\mylen}
   \setbox1=\hbox{$\bullet$}\setbox2=\hbox{\tiny$\bullet$}
   \setlength{\mylen}{\dimexpr0.5\ht1-0.5\ht2}
   \renewcommand\labelitemi{\raisebox{\mylen}{\tiny$\bullet$}}
   \begin{itemize}
     \item all extrema
     \item all zeros of $M(x)$ (366\,567\,325 in total)
     \item all values of $M(x)$ for $x$ a multiple of $10^8$
   \end{itemize}
   are reported.
   
   Finally, an algorithm is discussed that was used to compute $M(2^n)$ for all positive integers $n \leq 73$, including
   $$ M(2^{73}) = -6524408924. $$
   
   %Finally, an algorithm is discussed that was used to compute $$ M(10^{23}) = \texttt{XXX}, $$
   %improving the results of Kuzmenov, who computed $M$ at powers of $10$ up to $10^{22}$ \nolinebreak\cite{EK}.
   
   Section 2 describes the sieve used to compute $M(n)$ for all $n \leq 10^{16}$ and used in the main algorithm in the subsequent section.
   Section 3 derives a formula and incorporates it into an algorithm used to calculate $M(x)$ at an isolated value.
   Section 4 discusses the machinery used to derive bounds on $|q(x)|$. This entails analytic formulas relating to $M(x)$ and a lattice basis reduction scheme.
   Section 5 discusses all implementation details, which include low level tricks to speed up common calculations and the choice of hardware specific parameters.
   Section 6 presents and discusses the results of the computations. These include extrema of $M(x)$, properties of the zeros of $M(x)$, the values of $M(x)$ at isolated values, and various bounds on $q(x)$.
   Finally, section 7 summarizes all results and considers possible extensions.

   \section{Sieving Algorithm}
   \label{sieve}
   
   \noindent The functions $\mu(n)$ and $M(n)$ can be computed naively for all $n \leq x$ as follows \cite{DR}:
   \iffalse
   \begin{enumerate}[label=\arabic*.]
     \item Compute and store all primes $p \leq \sqrt{x}$.
     \item Initialize an array $m$ of $1$'s of length $\lfloor x \rfloor$.
     \item For each prime $p$
       \begin{itemize}[leftmargin=*]
       \item Multiply each value in $m$ whose index is a multiple of $p$ by \nolinebreak $-p$.
       \item Set each value in $m$ whose index is a multiple of $p^2$ to $0$.
       \end{itemize}
     \item For each $n \leq \lfloor x \rfloor$
       \begin{itemize}[leftmargin=*]
       \item If $m[n] = 0$, do nothing.
       \item If $|m[n]| = n$, set $m[n] = \text{sign}(m[n])$.
       \item Otherwise set $m[n] = -\text{sign}(m[n])$.
       \end{itemize}
     \item The array $m$ now stores $\mu(n)$ at position $n$.
     \item Cumulatively add the values in $m$ into another array. This array stores $M(n)$ at position $n$.
   \end{enumerate}
   \fi
   \vspace{-10pt}
   \begin{algorithm}
     \SetArgSty{}
     \DontPrintSemicolon
     Compute and store all primes $p \leq \sqrt{x}$\;
     Initialize an array $m$ of $1$'s of length $\lfloor x \rfloor$\;
     \For{each prime $p \leq \sqrt{x}$}{
       For all $1 \leq n \leq x$ divisible by $p$, set $m[n] \leftarrow -p \cdot m[n]$\;
       For all $1 \leq n \leq x$ divisible by $p^2$, set $m[n] \leftarrow 0$\;
     }
     \For{$1 \leq n \leq x$}{
       If $m[n] = 0$, do nothing\;
       If $|m[n]| = n$, set $m[n] \leftarrow \text{sign}(m[n])$\;
       Otherwise, set $m[n] \leftarrow -\text{sign}(m[n])$\;
     }
     The array $m$ now stores $\mu(n)$ at position $n$\;
     Cumulatively add the values in $m$ into another array. This array stores $M(n)$ at position $n$\;
   \end{algorithm}
   \vspace{-10pt}

   \noindent The runtime complexity of this sieve is determined by the first loop and is $$ O\bigg( \sum_{p \leq \sqrt{x}} \left( \frac{x}{p} + \frac{x}{p^2} \right) \bigg) = O(x \log\log x). $$
   
   %While this sieve is effective, there are performance issues that makes this algorithm impractical for large enough $x$. 
   There are two problems that render this algorithm impractical for large $x$.
   The first is that it requires $O(x \log\log x)$ multiplications, which can be costly. 
   The second is that the array $m$ must contain integers rather than bytes, which is less cache friendly. 
   The problem of cache misses is discussed in further detail in section \ref{implementation}.
   To address these issues a variation of this algorithm, similar to the one described in \cite{EK}, is used.
   Define $\theta(x)$ as the unit step function and $lsb(x)$ as the least significant bit of $x$, and sieve as follows:
   \iffalse
   \begin{enumerate}[label=\arabic*.]
     \item Create a byte-array $l$ of length $\lfloor \sqrt{x} \rfloor$ and set $l[j] = \lfloor \log_2 p_j \rfloor | 1$, where $p_j$ is the $j$th prime and $|$ is bitwise OR.
     \item Create a byte-array $m$ of length $\lfloor x \rfloor$ and set each element equal to 0\texttt{x}80 (set the most significant bit to 1 and the rest to 0).
     \item For each $j \leq \lfloor \sqrt{x} \rfloor$
       \begin{itemize}[leftmargin=*]
       \item Add $l[j]$ to each value in $m$ corresponding to a multiple of $p_j$.
       \item Set each value in $m$ whose index is a multiple of $p_j^2$ to $0$.
       \end{itemize}
     \item For each $n \leq \lfloor x \rfloor$
       \begin{itemize}[leftmargin=*]
       \item If the leading bit in $m[n]$ is $0$, set $m[n] = 0$.
       \item If $m[n] < \lfloor \log_2 n \rfloor - 5 - 2\theta(n-2^{20})$, set $m[n] = 2lsb(m[n]) - 1$.
       \item Otherwise set $m[n] = 1 - 2lsb(m[n])$.
       \end{itemize}
     \item The array $m$ now stores $\mu(n)$ at position $n$.
     \item Cumulatively add the values of $m$ to compute $M(n)$.
   \end{enumerate}
   \fi
   \begin{algorithm}
     \SetArgSty{}
     \DontPrintSemicolon
     Create byte-arrays $l$ of length $\lfloor \sqrt{x} \rfloor$ and $m$ of length $\lfloor x \rfloor$\;
     \For{$1 \leq j \leq \sqrt{x}$}{
       $l[j] \leftarrow \lfloor \log_2 p_j \rfloor | 1$, where $p_j$ is the $j$th prime and $|$ is bitwise OR\;
     }
     \For{$1 \leq n \leq x$}{
       $m[n] \leftarrow 0\texttt{x}80$ (set the most significant bit to 1 and the rest to 0)\;
     }
     \For{$1 \leq j \leq \sqrt{x}$}{
       For all $1 \leq n \leq x$ divisible by $p_j$, set $m[n] \leftarrow l[j] + m[n]$\;
       For all $1 \leq n \leq x$ divisible by $p_j^2$, set $m[n] \leftarrow 0$\;
     }
     \For{$1 \leq n \leq x$}{
       If the leading bit in $m[n]$ is $0$, set $m[n] \leftarrow 0$\;
       If $m[n] < \lfloor \log_2 n \rfloor - 5 - 2\theta(n-2^{20})$, set $m[n] \leftarrow 2lsb(m[n]) - 1$\;
       Otherwise, set $m[n] \leftarrow 1 - 2lsb(m[n])$\;
     }
   \end{algorithm}

   The idea of this algorithm is the same as the first one, except it works in log-space. 
   This allows multiplication to be replaced with addition and data to be stored in byte-arrays. 
   Though the time complexity remains the same, these changes reduce implementation overhead.
   %This provides a speedup by reducing implementation overhead, since the time complexity remains the same.
   
   After the third loop, the leading bit of each element $m[n]$ indicates whether $n$ is divisible by a square. This leaves 7 bits in $m[n]$ to add logarithms.
   Fortunately for all $n \leq 10^{16}$, the maximum possible amount of logarithms that can be added will not overflow to the eighth bit.
   In fact overflow will not occur until about $n = 10^{30}$. The least significant bit of each element $m[n]$ counts the parity of the number of prime factors encountered. If it is 0 there were an even amount and if it is 1 there were an odd amount. 
   This is achieved by setting the least significant bit in each element of $l$ to 1.
   
   Finally, logarithms are summed to determine if $n$ has a prime factor larger than $\sqrt{n}$ that was not accounted for in the sieve. 
   For $n \leq 2^{20}$, all primes will be accounted for if and only if $\sum_j \lfloor \log_2 p_j \rfloor | 1 < \lfloor \log_2 n \rfloor - 5$, 
   where all cases can be verified exhaustively. The validity for larger $n$ is shown by the following theorem.
   \begin{theorem}
   If $2^{20} < n \leq 10^{16}$, and $n = p_1 \cdots p_k$ is square-free, then
   \begin{equation} \label{1}
   \sum_{j = 1}^k \lfloor \log_2 p_j \rfloor | 1 \geq \lfloor \log_2 n \rfloor - 7
   \end{equation}
   and
   \begin{equation} \label{2}
   \sum_{j = 1}^{k-1} \lfloor \log_2 p_j \rfloor | 1 < \lfloor \log_2 n \rfloor - 7 \;\; \text{ when } \;\; p_k > \sqrt{n}.
   \end{equation}
   \end{theorem}
   \begin{proof}
   To show (\ref{1}) is true, a value of $n$ is sought that gives a sum which deviates below $\log_2 n$ as far as possible. 
   This will happen when there are many prime factors (allowing for more error), all $\lfloor \log_2 p_j \rfloor$ are odd (so the bitwise OR won't increment the sum), 
   and each $p_j$ is just less than a power of 2 (making the fractional part as large as possible). 
   Under these constraints there are a manageable number of cases to test manually.
   The largest deviation from $\lfloor \log_2 n \rfloor$ is $-7$ and first occurs at 
   $$ n = 3\cdot11\cdot13\cdot53\cdot59\cdot61\cdot229\cdot241\cdot251 \approx 1.13\cdot10^{15}. $$ 
   Additionally, the first occurrence of $-8$ is at 
   $$ n = 3\cdot13\cdot47\cdot53\cdot59\cdot61\cdot229\cdot239\cdot241\cdot251 \approx 1.16\cdot10^{18}, $$ 
   which means this algorithm will need to be slightly modified to reach that value.
   
   To show (\ref{2}) is true, observe 
    \begin{align*}
     \sum_{j=1}^{k-1} \lfloor \log_2 p_j \rfloor | 1 &\leq \sum_{j=1}^{k-1} \lfloor \log_2 p_j \rfloor + k-1 \\
             & \leq \log_2 n + k-1 - \log_2 p_k \\
             & \leq \lfloor \log_2 n \rfloor - 7 + (k + 6 - \log_2 \sqrt{n}\,).
   \end{align*}
   Now $$ \log_2 \sqrt{n} = \sum_{j=1}^{k} \log_2 \sqrt{p_j}, $$ and $\log_2 \sqrt{p_j} > 2$ for $j \geq 7$. 
   This leaves only a finite number of cases where $k + 6 < \log_2 \sqrt{n}$ might be false.
   Checking (\ref{2}) manually on each of these cases confirms its validity.
   \end{proof}
   
   Finally this sieving algorithm can be segmented into blocks small enough for a computer to store all generated data in RAM.
   Using a block size $B$ that is a divisor of $x$, compute $\mu(n)$ and $M(n)$ for all $(j-1)x/B + 1 \leq n \leq j x/B$ and let $j$ span from $1$ to $B$.
   For each block, only the primes up to $\sqrt{j x/B}$ need to be considered.

   \section{Combinatorial Algorithm}
   \label{combinatorial}

   To compute $M(x)$ at an isolated value, just as in \cite{EK} and \cite{DR}, start with the identity
   
   $$ \sum_{n \leq x} M(\lfloor x/n \rfloor) = 1. $$
   Observing $\lfloor x/n \rfloor$ takes on roughly $2\sqrt{x}$ distinct values, let $\nu_x = \nolinebreak \lfloor \sqrt{x}\, \rfloor$, $\kappa_x = \lfloor x/(\nu_x + 1) \rfloor$ and rewrite the identity as
   \begin{align*}
     \sum_{n \leq \kappa_x} M(x/n) &= 1 - \sum_{n \leq \nu_x} \left( \left\lfloor \frac{x}{n} \right\rfloor - \left\lfloor \frac{x}{n+1} \right\rfloor \right) M(n) \\
             &= 1 + \kappa_x M(\nu_x) - \sum_{n \leq \nu_x} \left\lfloor \frac{x}{n} \right\rfloor \mu(n).
   \end{align*}
   From an implementation standpoint, the second line is more cache friendly since the values of $\mu$ can be stored in an array of bytes.
   Moreover when $\mu(n) = 0$, the quotient it is multiplied by does not need to be computed.
      
   For any $\nu_x < u < x$ define
   $$ S(y, u) = 1 - \sum_{y/u < n \leq \kappa_y} M(y/n) + \kappa_y M(\nu_y) - \sum_{n \leq \nu_y} \left\lfloor \frac{y}{n} \right\rfloor \mu(n), $$
   which gives
   $$ \sum_{n \leq x/u} M(x/n) = S(x, u). $$
   Applying generalized M\"obius inversion yields the following result.
   \begin{theorem}
   \label{theorem3}
   \begin{align*}
     M(x) &= \sum_{n \leq x/u} \mu(n) S(x/n, u).
   \end{align*}
   \end{theorem}
   Now notice when computing this summand for all $n \leq x/u$, only the \emph{square-free} $n$ need to be considered, as $\mu(n) = 0$ otherwise.
   This means that about $1-6/\pi^2 \approx 39\% $ of summands need not be computed.
   
   To find each sum within each $S$, a segmented sieve can be applied to compute all required values of $\mu$ and $M$.
   The time complexity of this algorithm is thus the time spent sieving plus the time computing each $S(x/n, u)$.
   This gives a total time complexity of $$ O\bigg(u^{1+\varepsilon} + \sum_{n \leq x/u} \nu_{x/n}\bigg) = O(u^{1+\varepsilon} + x/\sqrt{u}\,). $$
   The choice of $u = O(x^{2/3+\varepsilon})$ minimizes this runtime complexity at $O(x^{2/3+\varepsilon})$.
   When performing the sieve, a sieving block size of $O(\sqrt{u}\,)$ can be used to obtain space complexity of $O(x^{1/3+\varepsilon})$.

   \section{Analytic Algorithm}
   \label{analytic}
   
   The bounds on $\liminf q(x)$ and $\limsup q(x)$ can be extended using the approach of Odlyzko and te Riele in \cite{OR},
   which begins with the following observation.
   \begin{theorem}[Tichmarsh \cite{T}]
   \label{theorem4}
   Assuming the Riemann hypothesis and all zeros of the zeta function are simple, then for $x > 0$,
   $$ M(x) = \sum_{i = 1}^\infty \bigg( \frac{x^{\rho_i}}{\rho_i\zeta'(\rho_i)} + \frac{x^{\overline{\rho_i}}}{\overline{\rho_i}\zeta'(\overline{\rho_i})} \bigg) + R(x) + \sum_{n = 1}^\infty \frac{(-1)^{n-1}(2\pi/x)^{2n}}{(2n)!n\zeta(2n+1)}. $$
   Here $R(x) = -2$ for $x \not\in \mathbb{Z}$, $R(x) = -2 + \mu(x)/2$ for $x \in \mathbb{Z}$, and $\rho_i$ is the $i$th non trivial zero of $\zeta$ with positive imaginary part.
   \end{theorem}
   Grouping terms in this formula gives 
   \begin{equation} \label{3}
     q(x) = 2\sum_{i=1}^\infty a_i \cos(\gamma_i \log x + \psi_i) + O(x^{-1/2}),
   \end{equation}
   where $a_i = 1/|\rho_i \zeta'(\rho_i)|$, $\gamma_i = \text{Im}(\rho_i)$, and $\psi_i = \arg(\rho_i \zeta'(\rho_i))$.
   Now defining $f(t) = (1-t)\cos(\pi t) + \sin(\pi t)/\pi$ and
   $$ h(y, N) = 2\sum_{i = 1}^N a_i f\left( \gamma_i/\gamma_N \right) \cos(\gamma_i y + \psi_i), $$
   the following holds.
   \begin{theorem}[Ingham \cite{I}]
   For any real $y$ and any positive integer $N$, $$ \liminf q(x) \leq h(y, N) \leq \limsup q(x). $$
   \end{theorem}
   One should note that unlike Theorem \ref{theorem4}, this theorem does not assume the Riemann hypothesis. Additionally this is the main result the analytic algorithm depends on.
   Roughly speaking, a trick to bound $q(x)$ is hence finding a $y$ and $N$ such that $|h(y, N)|$ is large. 
   Moreover since $\sum_i a_i$ diverges and $f(t) > 0$ for $0 < t < 1$, if all $\gamma_i y + \psi_i$ were close to multiples of $2\pi$ then $h(y, N)$ could be an arbitrarily large positive number.
   Similarly if all $\gamma_i y + \psi_i + \pi$ were close to multiples of $2\pi$, then $h(y, N)$ could be an arbitrarily large negative number.
   
   More explicitly, for any sequence of integers $m_i$ where $\cos(\gamma_i y + \psi_i - 2\pi m_i)$ is sufficiently small, $h(y, N)$ can be approximated with
   \begin{align*}
     h(y, N) &\approx 2\sum_{i = 1}^N a_i  \cos(\gamma_i y + \psi_i) \\
                &= 2\sum_{i = 1}^N a_i  \cos(\gamma_i y + \psi_i - 2\pi m_i) \\
                &\approx 2\sum_{i = 1}^N a_i - \sum_{i = 1}^N \bigg(\! \sqrt{a_i} (\gamma_i y + \psi_i - 2\pi m_i) \!\bigg)^2.
   \end{align*}
   This means if $m_i$ were found such that each $\sqrt{a_i} (\gamma_i y + \psi_i - 2\pi m_i)$ is small, $h(y, N)$ should be large.
   This can be achieved \emph{via} lattice reduction. Lattice reduction takes in a basis of integer vectors 
   and returns a new integer basis spanning the same space, where each vector has a small Euclidean norm.
   Fixing $N$, the initial basis is 
   $$
    \begin{pbvec}{cccccc} -\lfloor \sqrt{a_1}\psi_1 2^\nu \rfloor  \\ -\lfloor \sqrt{a_2}\psi_2 2^\nu \rfloor \\ \vdots \\ -\lfloor \sqrt{a_N}\psi_N 2^\nu \rfloor  \\ 2^\nu N^4 \\ 0 \end{pbvec},
    \begin{pbvec}{cccccc} \lfloor \sqrt{a_1}\gamma_1 2^{\nu-10} \rfloor  \\ \lfloor \sqrt{a_2}\gamma_2 2^{\nu-10} \rfloor \\ \vdots \\ \lfloor \sqrt{a_N}\gamma_N 2^{\nu-10} \rfloor  \\ 0 \\ 1 \end{pbvec},
    \begin{pbvec}{cccccc} \lfloor 2\pi\sqrt{a_1}2^\nu \rfloor  \\ 0 \\ \vdots \\ 0  \\ 0 \\ 0 \end{pbvec},
    $$
    $$
    \begin{pbvec}{cccccc} 0 \\ \lfloor 2\pi\sqrt{a_2}2^\nu \rfloor  \\ \vdots \\ 0  \\ 0 \\ 0 \end{pbvec},
    \cdots,
    \begin{pbvec}{cccccc} 0 \\ 0  \\ \vdots \\ \lfloor 2\pi\sqrt{a_N}2^\nu \rfloor \\ 0 \\ 0 \end{pbvec}
   $$
   where $\nu$ is any integer satisfying $2N \leq \nu \leq 4N$.
   
   Since $2^\nu N^4$ is much larger than every other element and no other vector has a nonzero $(N+1)$st component, 
   there should be exactly one reduced vector with a nonzero $(N+1)$st term and it will equal $\pm 2^\nu N^4$.
   Call this vector $v = (v_1, v_2, \ldots, v_{N+2})^\intercal$ and without loss of generality assume $v_{N+1} = 2^\nu N^4$.

   For each $1 \leq i \leq N$, this vector has components
   $$ v_i = z\lfloor \sqrt{a_i}\gamma_i 2^{\nu-10} \rfloor - \lfloor \sqrt{a_i}\psi_i 2^\nu \rfloor - m_i \lfloor 2\pi\sqrt{a_i}2^\nu \rfloor $$
   for some integers $z, m_1, m_2, \ldots, m_N$. 
   Now because $v_{N+1}$ is so large these terms should be small, which means
   $$ \sqrt{a_i}(\gamma_i z/2^{10} - \psi_i - 2\pi m_i) $$
   will also be small.
   Hence setting $y = z/2^{10}$ should give a value where $h(y, N)$ is large and positive,
   where the value $z$ is known, as $z = v_{N+2}$. 
   
   To find a $y$ that makes $h(y, N)$ large and negative, simply replace $\psi_i$ with $\psi_i + \pi$ in the call to the lattice reduction algorithm.
   
   Finally to improve results, the zeros $\rho_i$ can by sorted by $a_i$, rather than sorted by $\gamma_i$ as was done above.
   This will ensure the largest $a_i$'s will have their corresponding cosines near $\pm 1$, making the sum even larger.

   \section{Implementation Details}
   \label{implementation}
   
   \subsection{Sieve}
   
   When performing the sieve in section \ref{sieve}, the bottleneck is accounting for multiples of small prime powers, i.e. $2$, $3$, $2^2$, etc.
   To circumvent this, these values can be pre-sieved. This implementation pre-sieved with multiples of $2$, $3$, $2^2$, $5$, $7$, $3^2$, and $11$.
   To do this the sieve was applied, only using these numbers, on an array of length $2\cdot2\cdot3\cdot3\cdot5\cdot7\cdot11 = 13860$.
   When the main sieve was called, the array $m$ was assembled by joining many copies of this precomputed array.
   
   Because computing all $10^{16}$ values of $M$ at once would have required storing an array too large for RAM, the segmented version of the sieve was used.
   Computations were done in blocks of length $8\,728\,473\,600$, and used roughly 46 GB of RAM.
   During the main loop of the sieve, each block was further divided into smaller blocks to allow $m$ to fit in the L3 cache. % (on the machine used, $32$ MB).
   However, once the size of the primes became substantially larger than the length of $m$, too much time was spent iterating over primes that were never used.
   To address this, the length of $m$ was increased and no longer fit in the L3 cache.
   After each block was computed, each value of $M(n)$ was recorded if it was an extremum, zero, or if $n$ was a multiple of $10^8$.
   
   Finally, when identifying elements that correspond to a multiple of $p$ or $p^2$ in the sieve, integer division is required and is very costly.
   A way around this is to use methods described in \cite{GM}, which turns integer division into one 128 bit multiplication, one addition, and two bit shifts.
   This requires precomputing two constants for each denominator used in the scheme.
   
   \subsection{Combinatorial}
   
   To compute $M(x)$, the value $u = \lceil 0.5 x^{2/3} \rceil$ was chosen since it gave the fastest results.
   This means that when computing $M(2^{73})$, each $M(n)$ for all $n \leq 3.5\cdot10^{14}$ were computed through a segmented sieve.
   During this sieving process a block size of roughly $96\sqrt{2u}$ was chosen, giving a total of about $0.0073\sqrt{u}$ blocks to sieve through.
   Once a block of $\mu$ and $M$ values were computed, they were accounted for in each $S(x/n, u)$. Therefore all $S(x/n, u)$ were computed once the sieve finished.
   
   %The biggest bottleneck of this algorithm is integer division. 
   Computing all $S$ as stated in section \ref{combinatorial} requires $O(x^{2/3})$ integer divisions, and this is extremely costly.
   Fortunately when computing a value of $S$, both sequences of quotients that appear have the same numerator and each denominator successively increments by 1.
   This means all successive quotients $y/n$ with $\sqrt[3]{2y} \leq n \leq \sqrt{y}$ can be computed using a Bresenham style method. 
   This scheme computes a quotient based off the value of the previous quotient, and is described in detail in \cite{RS}.
   For all denominators $n < \sqrt[3]{2y}$, the same technique used in the sieve to turn a quotient into a multiplication, addition, and bit shifts can be employed \cite{GM}.
   The precomputing of constants for this method requires exactly one quotient to be computed per denominator.
   This reduces the number of integer divisions from $O(x^{2/3})$ to $O(x/u) = O(x^{1/3})$.
   
   \subsection{Analytic}
   
   Computing bounds on $q(x)$ requires many digits of $\rho_i$ and $\zeta'(\rho_i)$ and a fast lattice reduction routine.
   Mathematica was used to compute $\rho_i$ to $10\,000$ digits of precision for all $i \leq 14\,400$ and subsequently compute each $\zeta'(\rho_i)$ to roughly $8151$ digits of precision.
   The results were verified using PARI/GP.
   
   The lattice reduction library chosen was fplll \cite{FPLLL}.
   Its implementation has a runtime complexity of $O(N^{4+\varepsilon} \nu (N + \nu))$, which is faster than the original algorithm's runtime complexity of $O(N^{6+\varepsilon} \nu^3)$ \cite{MSV}.
   For each call to fplll, the optional parameter values $(\delta, \eta) = (0.9999, 0.99985)$ were used.
   The choices of these parameters were intended to speed up the runtime, with the trade off of a less optimal solution.
   
   \subsection{Hardware}

   The computations of $\rho_i$ and $\zeta'(\rho_i)$ were performed on a $360$ core cluster on the Wrangler system at the Texas Advanced Computing Center.
   All other computations were run on a 2.7 GHz 12-core Intel Xeon E5 processor with a 32 MB L3 cache and 64 GB of RAM.
   The code was compiled with g\texttt{++} and where possible, routines were parallelized using OpenMP.

   \section{Results}
   \label{results}
   %\begin{center}
   %\includegraphics[scale = .3]{timingplot.pdf}
   %\end{center}
   \subsection{Sieve}
   
   Computing $M(n)$ for all $n \leq 10^{16}$ took roughly $7.5$ months and was heavily influenced by cache misses.
   The frequency of these misses increased with $n$.
   For comparison, the first $10^{14}$ values took $1$ day to compute, the next $10^{14}$ values took $1.35$ days, 
   and this gradually increased until the final $10^{14}$ values took $2.8$ days. 
   Results were periodically verified throughout the computation using the algorithm described in \cite{DR} to compute $M(n)$ and compare values.
   No discrepancies were found.
   
   \begin{center}
   \includegraphics[scale = .33]{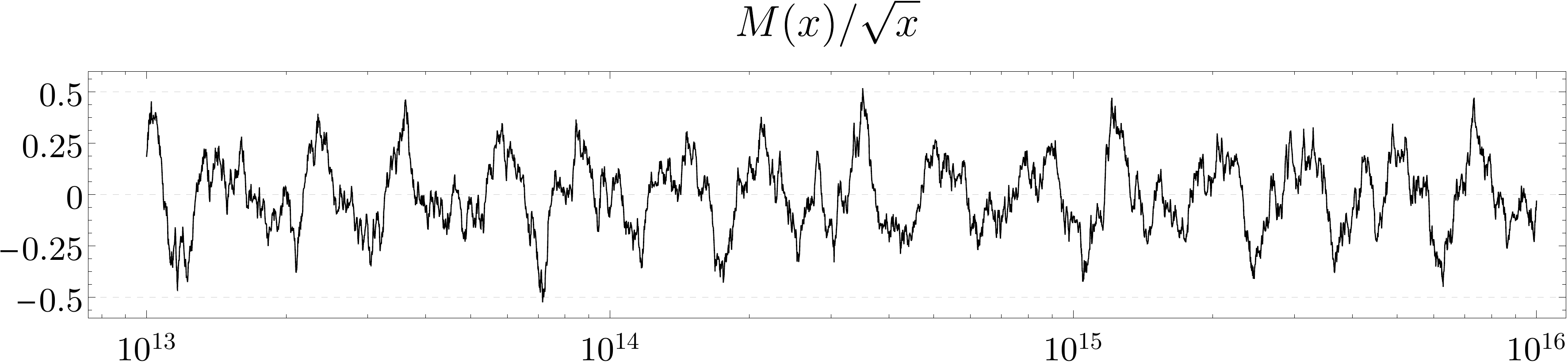}
   \end{center}
   
   The largest absolute values $M(n)$ attains for $n \leq 10^{16}$ are 
   $-35\,629\,003$ and $40\,371\,499$,
   and the largest absolute values $q(n)$ attains in this interval are
   $-0.525$ and $0.571$.
   Below is a select list of extrema corresponding to prominent peaks of $M$:
   
   {\fontsize{7}{8}\selectfont
   \begin{center}
   \begin{tabular}{|rrr||rrr|}
   \hline
 $n\quad\quad\;\;\;$ & $M(n)$ & $q(n)$ & $n\quad\quad\quad\quad$ & $M(n)\;\;$ & $q(n)$ \\\hline
 6631245058 & -31206 & -0.383 & 5197159385733 & -689688 & -0.303 \\
 \bf7766842813 & 50286 & \bf0.571 & 10236053505745 & 1451233 & 0.454 \\
 15578669387 & -51116 & -0.410 & 21035055623987 & -1740201 & -0.379 \\
 19890188718 & 60442 & 0.429 & 21036453134939 & -1745524 & -0.381 \\
 22867694771 & -62880 & -0.416 & 23431878209318 & 1903157 & 0.393 \\
 38066335279 & -81220 & -0.416 & 30501639884098 & -1930205 & -0.349 \\\hline
 48638777062 & 76946 & 0.349 & 36161703948239 & 2727852 & 0.454 \\
 56808201767 & -87995 & -0.369 & 36213976311781 & 2783777 & 0.463 \\
 101246135617 & -129332 & -0.406 & \bf71578936427177 & -4440015 & \bf-0.525 \\
 108924543546 & 170358 & 0.516 & 146734769129449 & 3733097 & 0.308 \\
 148491117087 & -131461 & -0.341 & 175688234263439 & -5684793 & -0.429 \\
 217309283735 & -190936 & -0.410 & 212132789199869 & 5491769 & 0.377 \\\hline
 297193839495 & 207478 & 0.381 & 212137538048059 & 5505045 & 0.378 \\
 330508686218 & -294816 & -0.513 & 304648719069787 & -5757490 & -0.330 \\
 402027514338 & 271498 & 0.428 & 351246529829131 & 9699950 & 0.518 \\
 661066575037 & 331302 & 0.407 & 1050365365851491 & -13728339 & -0.424 \\
 1440355022306 & -368527 & -0.307 & 1211876202620741 & 16390637 & 0.471 \\
 1653435193541 & 546666 & 0.425 & 2458719908828794 & -20362905 & -0.411 \\\hline
 2087416003490 & -625681 & -0.433 & 3295555617962269 & 18781262 & 0.327 \\
 2343412610499 & 594442 & 0.388 & 3664310064219561 & -23089949 & -0.381 \\
 3270926424607 & -635558 & -0.351 & 4892214197703689 & 24133331 & 0.345 \\
 4098484181477 & 780932 & 0.386 & \bf6287915599821430 & \bf-35629003 & -0.449 \\
 5191164528277 & -668864 & -0.294 & \bf7332940231978758 & \bf40371499 & 0.471 \\\hline
   \end{tabular}
   \end{center}
   }
   
   All zeros of $M(n)$ for $n \leq 10^{16}$ were recorded.
   A natural question to ask is for any $x$, how many zeros are less than $x$?
   Defining $V(x)$ to be the number of zeros less than $x$, a theorem of Landau states $V(x) = \Omega(\log x)$ \cite{EG}.
   This however is expected to be a weak lower bound.
    
    Treating $M(n)$ as a random walk with probability of staying stationary $1-6/\pi^2$ and with both probabilities of moving up and down $3/\pi^2$, it would follow that $V(x) = \sqrt{\pi x/3} + o(\sqrt{x}\,)$.
    In practice however $M$ cannot be modeled as a random walk because there is regularity, e.g. $M(4n+3) = M(4n+4)$, etc.
    Nonetheless, the data suggest $V(x) = \Theta(x^{1/2 + \varepsilon})$.
    In fact $3.5 \sqrt{x}$ or even $\sqrt{x} \log \log x$ seem like good approximations.
    
   \begin{center}
   \begin{tabular}{cc}
   {\small
    \adjustbox{valign=t}{\begin{tabular}{l}\\
    \begin{tabular}{| ll || ll |} %pipe L L pipe pipe L L pipe ... -__-
     \hline
     $n\;\,$ & $\!\!\!\!V(10^n)$ & $n\phantom{1}\;\,$ & $V(10^n)$ \\
     \hline
     1 & \!\!\!\!1 & 9 & \!\!\!\!141121 \\
     2 & \!\!\!\!6 & 10 & \!\!\!\!431822 \\
     3 & \!\!\!\!92 & 11 & \!\!\!\!1628048 \\
     4 & \!\!\!\!406 & 12 & \!\!\!\!4657633 \\
     5 & \!\!\!\!1549 & 13 & \!\!\!\!12917328 \\
     6 & \!\!\!\!5361 & 14 & \!\!\!\!40604969 \\
     7 & \!\!\!\!12546 & 15 & \!\!\!\!109205859 \\
     8 & \!\!\!\!41908 & 16 & \!\!\!\!366567325 \\
     \hline
    \end{tabular}
    \end{tabular}}
    } & \adjustbox{valign=t}{\includegraphics[scale = .4]{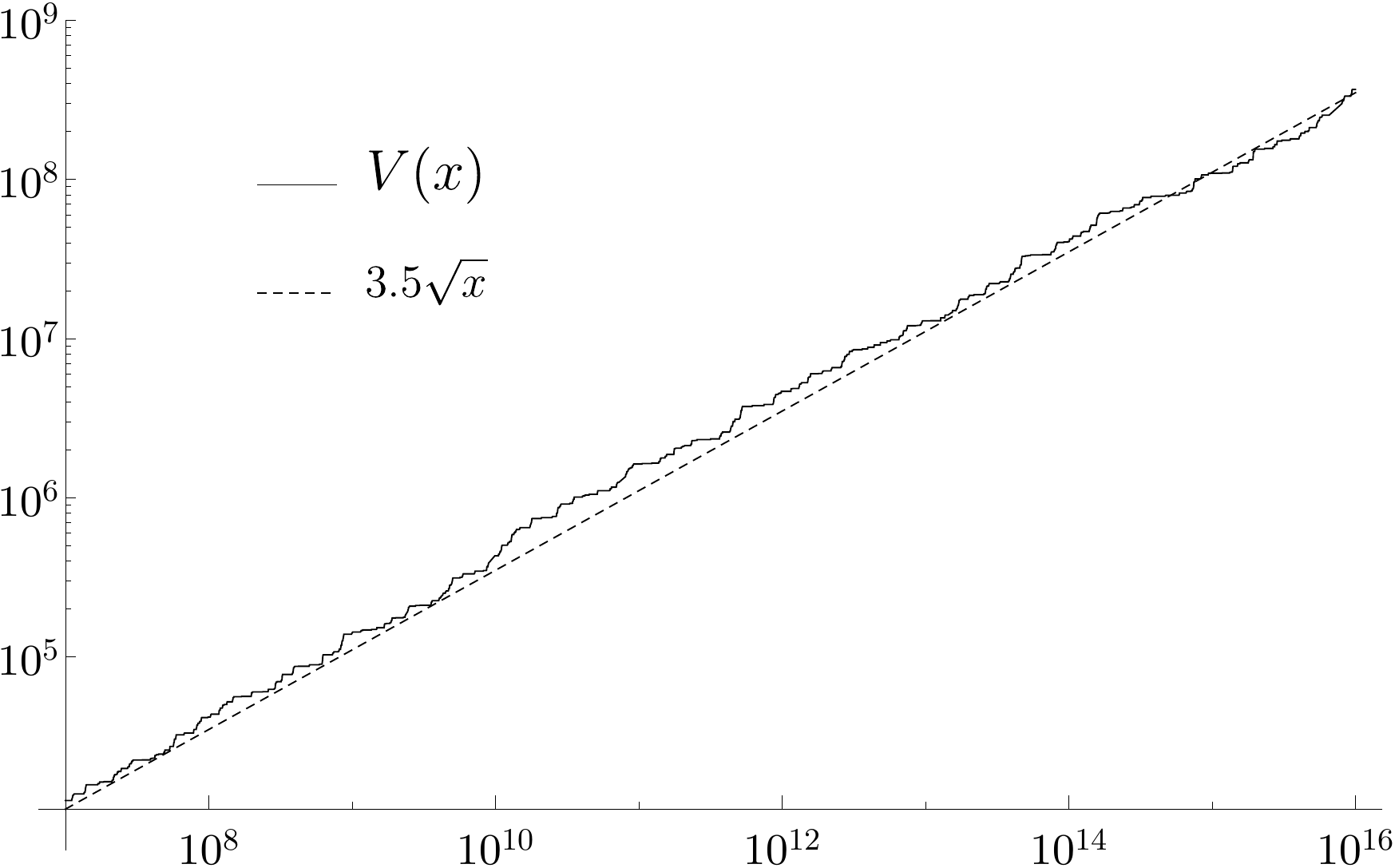}}
   \end{tabular}
   \end{center}
   
   A property these zeros can help investigate is whether $M$ tends to have a bias towards being positive or negative.
   Define $M_{\texttt{+}}(x)$ to be the percentage of $M(n)$ that are positive for $n \leq x$, that is
   $$ M_{\texttt{+}}(x) = \frac{1}{x} \sum_{n \leq x \atop M(n) > 0} 1. $$
   A direct consequence of work by Ng \cite{NG} is that under certain conjectures the average value of $M_{\texttt{+}}(x)$ should be $1/2$, i.e. no bias should exist.
   Computing $\mu$ at each zero of $M$, the sign of $M$ can be determined between consecutive zeros which can be used to compute $M_{\texttt{+}}$.
   For $x \leq 10^5$ there is a clear negative bias, but for $10^5 \leq x \leq 10^{16}$ there is no longer any apparent bias.
   For $10^5 \leq x \leq 10^{16}$ the extreme values are $M_{\texttt{+}}(53\,961\,131\,760\,658) \approx \nolinebreak0.385$ and $M_{\texttt{+}}(238\,469\,701\,201\,412) \approx 0.601$.\vspace{5pt}
   
   \begin{center}
   \includegraphics[scale = .46]{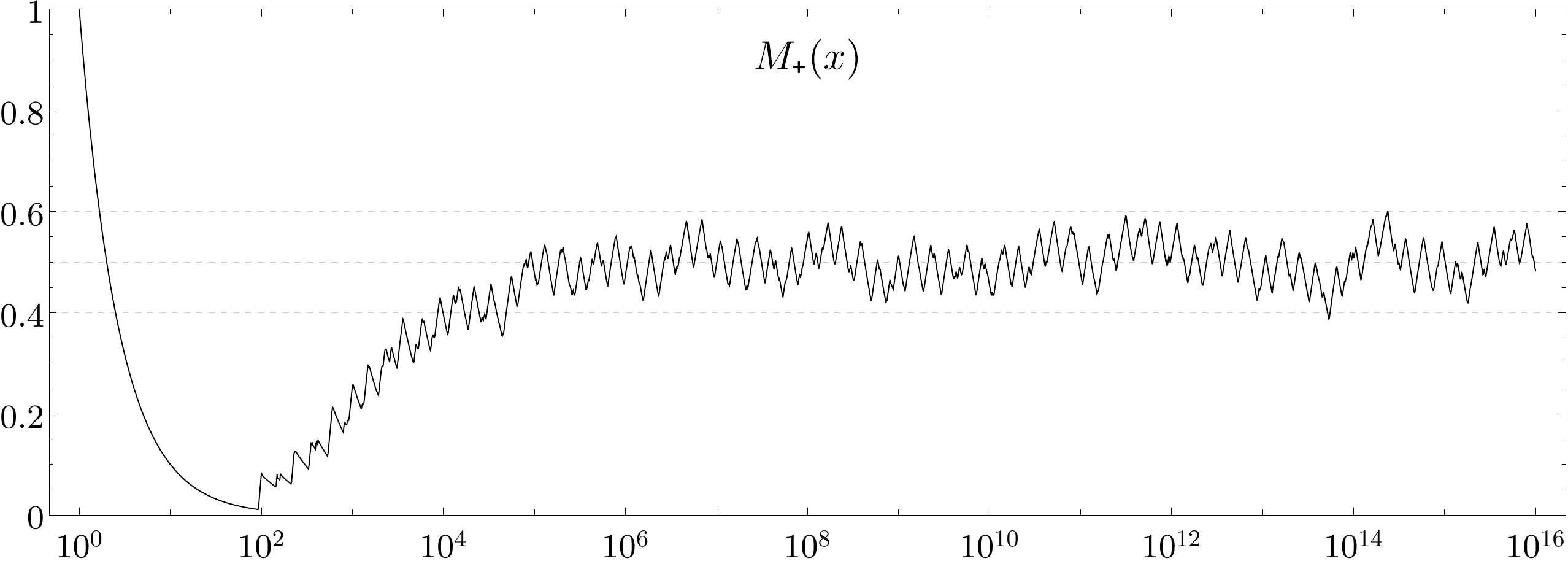}
   \end{center}
   
   Another characteristic of the zeros worth consideration is the gap between two consecutive zeros.
   To examine these gaps, let $G_m(g)$ be the number of gaps of length $g$ that occur for the first $m$ zeros.
   For a fixed value of $m$, this function can be plotted to show how the number of gaps of certain lengths vary.
   Letting $\omega = V(10^{16}) = 366\,567\,325$, gives the following plot:
   
   \begin{center}
   \includegraphics[scale = .57]{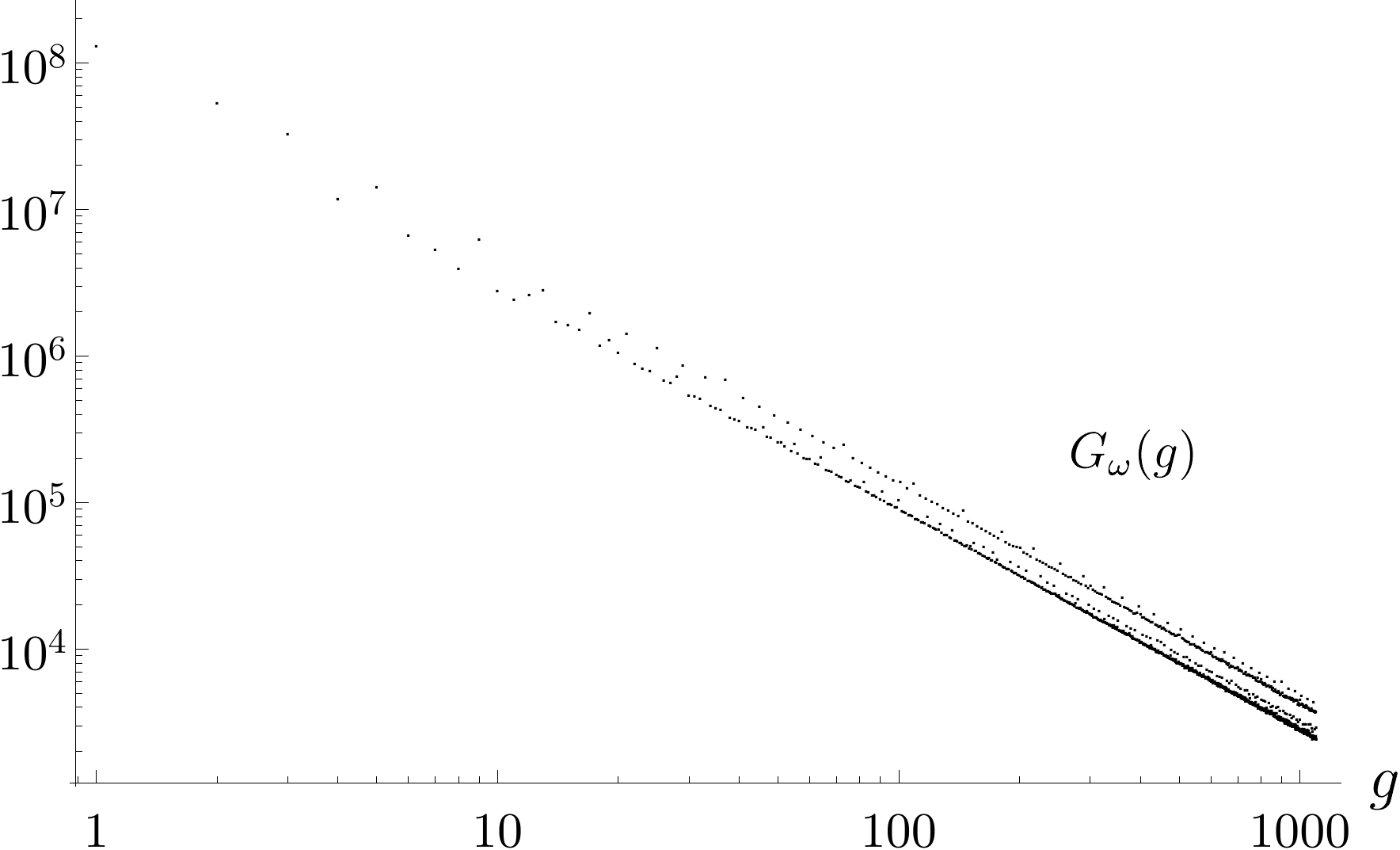}
   \end{center}
   
   As seen above, there are distinct bands present and each looks to roughly follow a power law, all with the same exponent.
   Zooming in, it appears each band is represented by all $g$ congruent to $1$ modulo a product of distinct primes squared.
   
   \begin{center}
   \includegraphics[scale = .75]{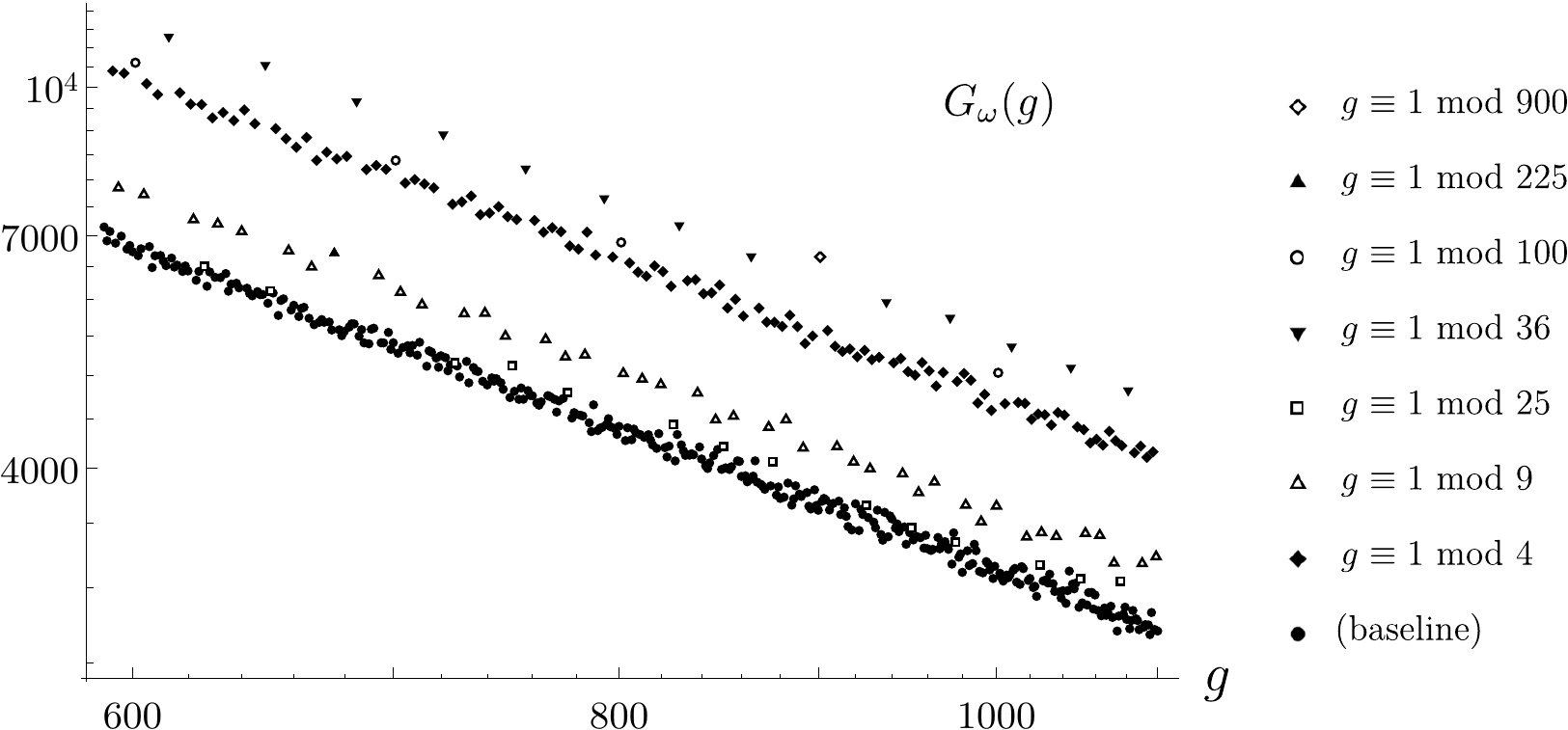}
   \end{center}
   Defining $b_m(g)$ to be the baseline band (which can be approximated by a power law) and $P_g$ to be the set of all primes $p$ where $g \equiv 1 \!\!\mod{p^2}$, it seems these bands are expressed with the multiplier
   $$ G_m(g) = \left( \sum_{S \in {\mathcal P}(P_g)} \prod_{p \in S} \frac{1}{p^2-2} \right) b_m(g). $$
   For example if $g_0 \equiv 1 \!\!\mod{4}$ and $g_0 \not\equiv 1 \!\!\mod{p^2}$ for $p \neq 2$, then $P_{g_0} = \{2\}$ and $G_m(g_0)$ should be above $b_m(g_0)$ by a multiplicative factor of $3/2$.
   
   \subsection{Combinatorial}

   Calculating $M(x)$ at powers of two scaled roughly as $O(x^{2/3})$, i.e. $M(2^{x+1})$ was about $2^{2/3} \approx 1.59$ times slower to compute than $M(2^x)$. 
   However, as in the sieve above, cache misses became more frequent for larger $x$ resulting in scale factors around $1.63$.
   The results are as follows:\vspace{5pt}
   
   {\fontsize{7}{8}\selectfont
   \begin{center}
   \begin{tabular}{|rr||rr||rr||rrr|}
   \hline
 $n$ & $M(2^{n^{\vphantom{n}}})$ & $n\,$ & $M(2^n)$ & $n\,$ & $M(2^n)$ & $n\,$ & $M(2^n)$ & time\,{(s)} \\\hline
 0 & $1$ & 19 & $\texttt{-}125$ & 38 & $38729$ & 57 & $51885062$ & $236.02$ \\
 1 & $0$ & 20 & $257$ & 39 & $\texttt{-}135944$ & 58 & $\texttt{-}15415164$ & $374.60$ \\
 2 & $\texttt{-}1$ & 21 & $\texttt{-}362$ & 40 & $101597$ & 59 & $\texttt{-}89014828$ & $594.65$ \\
 3 & $\texttt{-}2$ & 22 & $228$ & 41 & $15295$ & 60 & $\texttt{-}48425659$ & $943.63$ \\
 4 & $\texttt{-}1$ & 23 & $\texttt{-}10$ & 42 & $\texttt{-}169338$ & 61 & $220660381$ & $1494.41$ \\
 5 & $\texttt{-}4$ & 24 & $211$ & 43 & $259886$ & 62 & $\texttt{-}248107163$ & $2378.21$ \\\hline
 6 & $\texttt{-}1$ & 25 & $\texttt{-}1042$ & 44 & $\texttt{-}474483$ & 63 & $580197744$ & $3815.14$ \\
 7 & $\texttt{-}2$ & 26 & $329$ & 45 & $1726370$ & 64 & $\texttt{-}851764249$ & $6263.46$ \\
 8 & $\texttt{-}1$ & 27 & $330$ & 46 & $\texttt{-}3554573$ & 65 & $809210153$ & $10376.5$ \\
 9 & $\texttt{-}4$ & 28 & $\texttt{-}1703$ & 47 & $\texttt{-}135443$ & 66 & $\texttt{-}1220538763$ & $17235.2$ \\
 10 & $\texttt{-}4$ & 29 & $6222$ & 48 & $3282200$ & 67 & $\texttt{-}925696220$ & $28404.4$ \\
 11 & $7$ & 30 & $\texttt{-}10374$ & 49 & $1958235$ & 68 & $2092394726$ & $46429.7$ \\\hline
 12 & $\texttt{-}19$ & 31 & $9569$ & 50 & $\texttt{-}1735147$ & 69 & $\texttt{-}3748189801$ & $75680.8$ \\
 13 & $22$ & 32 & $1814$ & 51 & $6657834$ & 70 & $9853266869$ & $123189$ \\
 14 & $\texttt{-}32$ & 33 & $\texttt{-}10339$ & 52 & $\texttt{-}13927672$ & 71 & $\texttt{-}12658250658$ & $200574$ \\
 15 & $26$ & 34 & $\texttt{-}3421$ & 53 & $\texttt{-}11901414$ & 72 & $9558471405$ & $326068$ \\
 16 & $14$ & 35 & $8435$ & 54 & $48662015$ & 73 & $\texttt{-}6524408924$ & $529127$ \\
 17 & $\texttt{-}20$ & 36 & $38176$ & 55 & $\texttt{-}48361472$ & & & \\%$916777$ \\
 18 & $24$ & 37 & $\texttt{-}28118$ & 56 & $23952154$ &  &  & \\%$1481180$ \\
 \hline
   \end{tabular}
   \end{center}
   }
   
   %\begin{center}
   %\begin{tabular}{ | c | c | r | c | }
   %\hline
   %$10^n$ & $M(n)$ & $q(n)\hspace{3pt}$ & time\,{(s)}\\
   %\hline
   %$14$ & -875575 & -0.088 & 1.93\\
   %$15$ & -3216373 & -0.102 & 9.25\\
   %$16$ & -3195437 & -0.032 & 42.36\\
   %$17$ & -21830254 & -0.069 & 197.92\\
   %$18$ & -46758740 & -0.047 & 918.95\\
   %$19$ & 899990187 & 0.285 & 4316.07\\
   %$20$ & 461113106 & 0.046 & 23026.8\\
   %$21$ & -3395895277\footnotemark & -0.107 & 117254\\
   %$22$ & -2061910120 & -0.021 & 589112\\
   %$23$ & \texttt{----XXX----} & 0.202 & 2892540\\
   %\hline
   %\end{tabular}
   %\end{center}
   %\footnotetext{The corresponding value listed in \cite{EK} is missing the minus sign. This typo was confirmed by the author of \cite{EK} through email correspondence.}
   
   The correctness of the implementation was verified in 3 ways:
   \begin{itemize}
     \item Tests on many already known values were run.
     \item When computing $M(x)$, $M(x/128)$ was simultaneously computed.
     \item Formula (\ref{3}) was used to estimate the first couple digits of $M(x)$ and its order of magnitude.
   \end{itemize}
   \noindent All values were found to agree.
   %Tests on many already known values were run, when computing $M(x)$, $M(x/128)$ was simultaneously computed,
   %and the formula (\ref{3}) was used to estimate the first couple digits of $M(x)$ and its order of magnitude.
   
   \subsection{Analytic} The results of deriving bounds on $q(x)$ can be summarized with the following theorem.
   \begin{theorem}
   The function $q(x) = M(x)/\sqrt{x}$ has bounds 
   $$\liminf q(x) < -1.837625$$ and $$\limsup q(x) > 1.826054.$$
   \end{theorem}
   \begin{proof}
   To derive these bounds, the lattice reduction algorithm covered in section \ref{analytic} was run with inputs $\nu = 17\,000$ and $N = 800$.
   Both calls took roughly $35$ days to finish, giving $y$ values
   $$ y_{\texttt{-}} \approx 1.50546\cdot10^{5096} \quad \text{ and } \quad y_{\texttt{+}} \approx -2.58842\cdot10^{5097}, $$
   where their exact values can be found in the appendix below. Evaluating $h(y_\pm, 14400)$ gives the extreme values
   $$ h(y_{\texttt{-}}, 14400) \approx -1.837625 \quad \text{ and } \quad h(y_{\texttt{+}}, 14400) \approx 1.826054. $$
   \end{proof}
   In addition, the lattice reduction algorithm was run on various choices of smaller $\nu$ and $N$. These establish some weaker bounds:
   \begin{center}
   \begin{tabular}{ | c | c | r | c | c | }
   \hline
   $\nu$ & $N$ & $y\quad\quad\quad\;\;$ & $h(y, 14400)\;$ & time\,{(d)}\\
   \hline
   $5000$ & $400$ & $-2.78367\cdot10^{1493^{\vphantom{3}}}$ & $\phantom{-}1.61230$ & $0.53$\\
   $12000$ & $600$ & $-5.19605\cdot10^{3594}$ & $-1.76011$ & $7.32$\\
   $12000$ & $600$ & $9.31709\cdot10^{3594}$ & $\phantom{-}1.76382$ & $7.33$\\
   $15000$ & $700$ & $2.74696\cdot10^{4495}$ & $-1.81111$ & $19.00$\\
   $15000$ & $700$ & $9.69908\cdot10^{4495}$ & $\phantom{-}1.81252$ & $18.99$\\
   $17000$ & $800$ & $1.50546\cdot10^{5096}$ & $-1.83762$ & 35.07\\
   $17000$ & $800$ & $-2.58842\cdot10^{5097}$ & $\phantom{-}1.82605$ & 35.09\\
   \hline
   \end{tabular}
   \end{center}
   
   Finally, an approximate formula can be used to visualize what $q(x)$ might look like in the neighborhood of $y_\pm$.
   Defining $$ \tilde{q}(x) = 2\sum_{i=1}^{14400} a_i \cos(\gamma_i x + \psi_i) $$
   and assuming $\tilde{q}(x) \approx q(e^x)$ gives plots about these extreme values:
   
   \begin{center}
   \includegraphics[scale = .65]{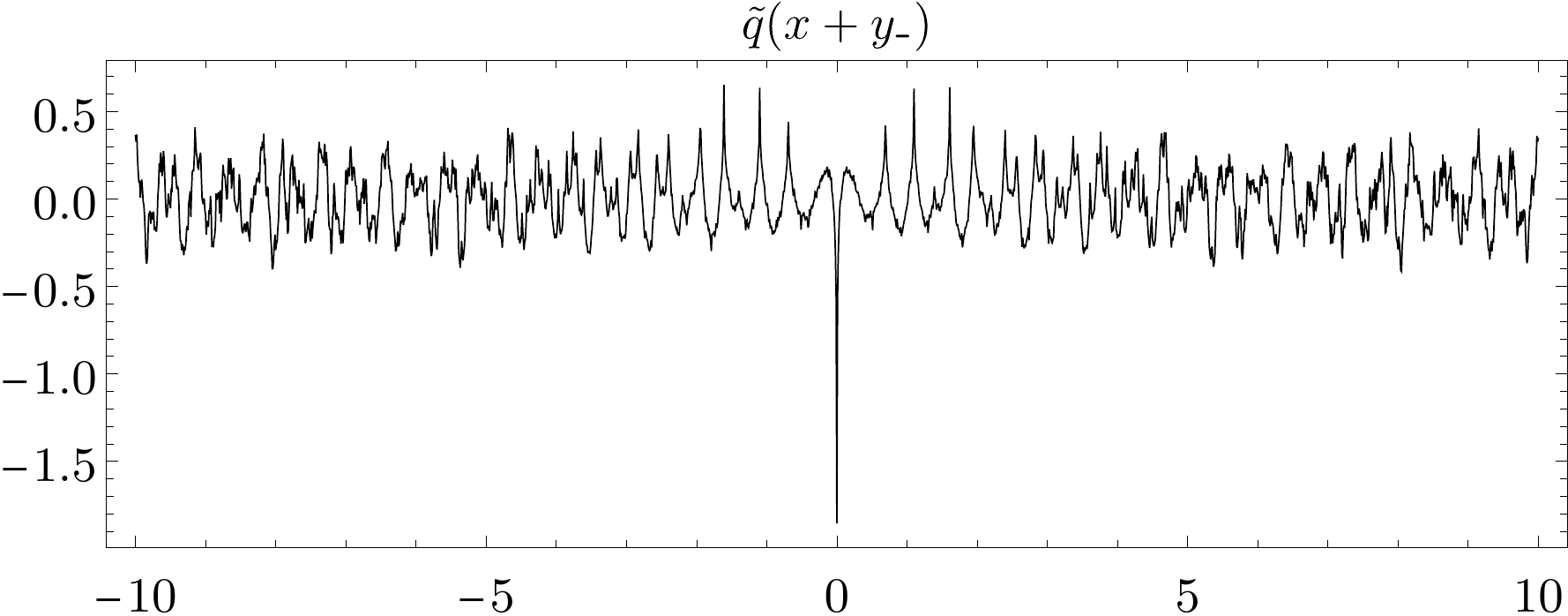}
   \end{center}
   \begin{center}
   \includegraphics[scale = .65]{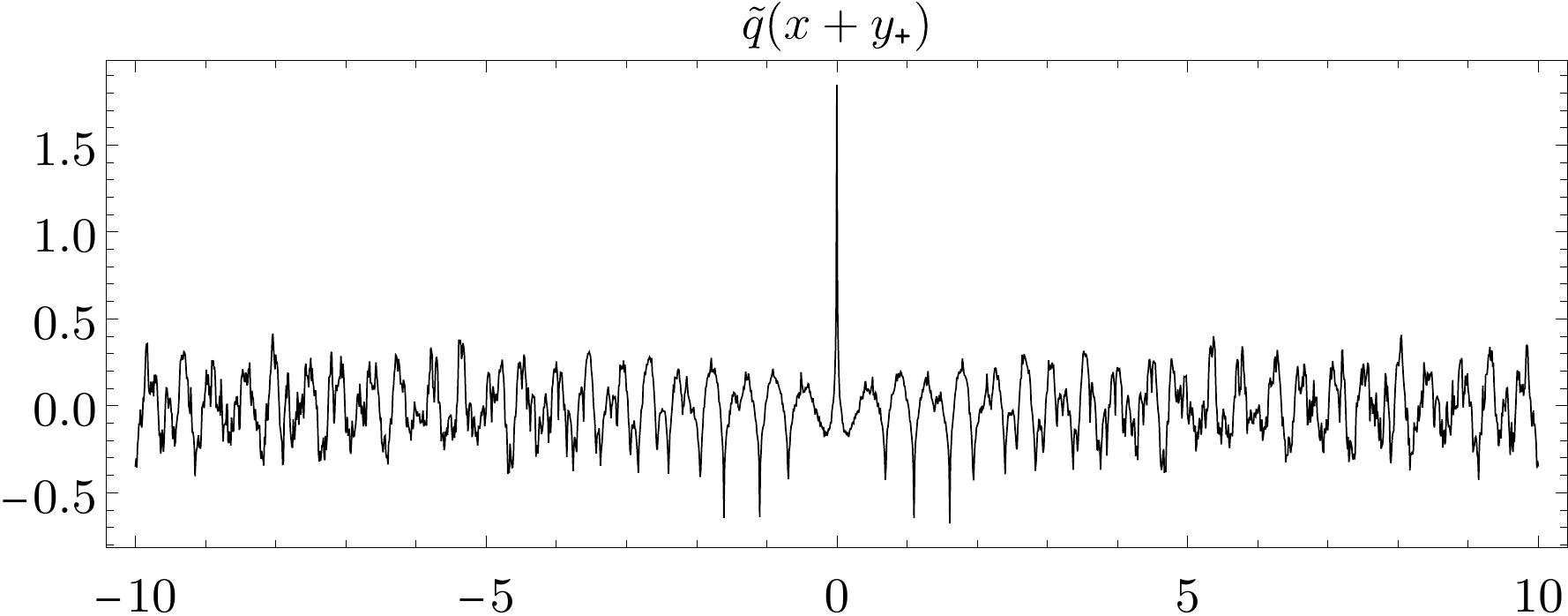}
   \end{center}
   The observation made in \cite{KJ} that the width of the peaks of $q(x)$ remain constant with respect to $\log x$ seems to hold this far out.
   Moreover as seen in the above figures, these peaks appear to be anomalies, as most peaks in the vicinity of $y_\pm$ do not exceed $0.5$ in absolute value.
   %It makes these peaks truly resemble a needle in a haystack.
   
   \section{Extensions and Concluding Remarks}
   
   \subsection{Sieve}
   Computing $M(n)$ for all $n \leq 10^{16}$ took about $7.5$ months and the time was dominated by cache misses.
   To systematically compute $M(n)$ for say $n \leq 2 \cdot 10^{16}$, the cache misses beyond $10^{16}$ would grow substantially more frequent, causing a drastic slow down.
   To reduce the number of these misses, additional measures can be taken.

   First, rather than storing each value $\mu(n)$ in $1$ byte, $4$ values of $\mu(n)$ can be encoded together since $\mu(n)$ only takes on $3$ possible values, allowing it to be expressed with $2$ bits.
   A similar approach can be taken for $M(n)$ too, but not for $n > 10^{16}$. For computations on shorter intervals though, space can still be saved.
   For example, $M(n)$ can be stored as a signed 16 bit integer as long as $|M(n)| < 2^{15}$. The first time this inequality is violated is at $n = \nolinebreak7\,613\,644\,886$.
   Similarly, $M(n)$ can be stored as a signed 24 bit integer for all $n < \nolinebreak348\,330\,855\,359\,510$.
   
   A more robust solution to prevent cache misses is to employ an additional data structure.
   Recall that during the sieve the array built to store values of $\mu$ is segmented into blocks small enough to fit into the L3 cache.
   However once the primes being iterated over become too large, much time is wasted iterating over primes that aren't used.
   Currently, this is mitigated by using larger blocks, but these larger blocks no longer fit in the L3 cache. 
   Instead, this problem could be resolved in the following way:
    \iffalse
    \begin{enumerate}[label=\arabic*.]
     \item Create a hashmap $h$ that maps an integer to a vector of integers.
     \item For each prime $p$, find the first block $i$ that contains an index corresponds to a multiple of $p$.
       \begin{itemize}[leftmargin=*]
       \item If $h(i)$ is uninitialized, set $h(i) = \{p\}$.
       \item Otherwise append $p$ to the vector $h(i)$.
       \end{itemize}
     \item When sieving over block $i$, only iterate over the primes in $h(i)$.
     \item In block $i$, once a prime $p$ has been iterated over
     \begin{itemize}[leftmargin=*]
       \item Determine the next block $j$ in which $p$ will be used.
       \item If $h(j)$ is uninitialized, set $h(j) = \{p\}$.
       \item Otherwise append $p$ to the vector $h(j)$.
       \end{itemize}
     \item Clear $h(i)$ once sieving block $i$ has finished.
   \end{enumerate}
   \fi
   \begin{algorithm}
     \SetArgSty{}
     \DontPrintSemicolon
     Create a hashmap $h$ that maps an integer to a vector of integers\;
     \For{each prime $p \leq \sqrt{x}$}{
       Find the first block $i$ with an index corresponding to a multiple of $p$\;
       If $h(i)$ is uninitialized, set $h(i) \leftarrow \{p\}$\;
       Otherwise append $p$ to the vector $h(i)$\;
     }
     \For{each block $i$}{
       \For{each $p$ in $h(i)$}{
         Sieve block $i$ with $p$ as normal\;
         Determine the next block $j$ in which $p$ will be used\;
         If $h(j)$ is uninitialized, set $h(j) \leftarrow \{p\}$\;
         Otherwise append $p$ to the vector $h(j)$\;
       }
       Clear $h(i)$\;
     }
   \end{algorithm}
   
   Under this approach, the block size can be set to always fit in the L3 cache without having the overhead of iterating over primes that will never be used.
   Notice here that each prime $p$ will only be present in $h$ at most once.
   Hence the size of $h$ is only dependent on the number primes used, not the number of blocks being iterated over.
   For an L3 cache similar in size to the one used, this method could help make it feasible to compute beyond $10^{16}$.
         
   \subsection{Combinatorial}
   Isolated values of $M(x)$ were computed at powers of $2$ up to $M(2^{73}) = -6524408924$, which took roughly $6$ days to calculate.
   At the time this paper was written, to the author's knowledge, there are no known combinatorial identities that lead to a runtime complexity less than $O(x^{2/3+\varepsilon})$.
   However a speedup could still potentially be obtained with a combinatorial approach.
   
   Recall the identity used in the algorithm and stated in Theorem \ref{theorem3} is
   $$ M(x) = \sum_{n \leq x/u} \mu(n) S(x/n, u). $$
   Since $\mu(n)$ will asymptotically be zero $1 - 6/\pi^2 \approx 39\%$ of the time,
   one approach could be to look for a sum whose summand is zero more often than this.
   The closest identity the author found in literature is due to Benito and Varona \cite{BV} and is
   $$ M(x) = \frac{1}{2}\sum_{n \leq x/u} f^{-1}(n) G(x/n, u), $$
   where
   \begin{align*}
     G(y, u) = -&3 + \sum_{y/u < n \leq \kappa_y} (h(n)-h(n-1)) M(y/n) + h(\nu_y) M(\kappa_y) \\
     & + \sum_{n \leq \nu_y} \bigg( 3\bigg\lfloor \frac{n}{3k} \bigg\rfloor - 2\bigg\lfloor \frac{n-k}{2k} \bigg\rfloor \bigg) \mu(n),
   \end{align*}
   and $f^{-1}(n)$ is the Dirichlet inverse of $f(n) = h(n-1)-h(n)$, and \vspace{-3pt}
   {\fontsize{10}{12}\selectfont $$ h(n) = \begin{cases}\vspace{-3pt}2 & \mbox{ if } n \equiv 0 \!\!\! \mod 6\\\vspace{-3pt} 0 & \mbox{ if } n \equiv 1 \text{ or } 2 \!\!\! \mod 6\\\vspace{-3pt}1 & \mbox{ if } n \equiv 3 \text{ or } 4 \!\!\! \mod 6\\-1 & \mbox{ if } n \equiv 5 \!\!\! \mod 6.\end{cases} $$}
   It turns out $f^{-1}(n)$ is zero just as often as $\mu(n)$ with the added advantage that $f^{-1}(2) = f^{-1}(4) = 0$, meaning $2$ of the $4$ most computationally expensive summands need not be computed.
   The drawback is that no efficient way of computing $f^{-1}(n)$ was found.
   
   Lastly, an analytic approach could be considered. In 1987 Lagarias and Odlyzko described a way to compute $\pi(x)$, the number of primes $\leq x$, in $O(x^{1/2 + \varepsilon})$ time \cite{LO}.
   The algorithm uses a completely different approach, expressing $\pi(x)$ in terms of a contour integral in the complex plane.
   Moreover the discussion section in \cite{LO} states that the same ideas can be applied to compute $M(x)$ in the same time complexity.
   
   In $2010$, Platt computed $\pi(x)$ using this algorithm and stated the combinatorial algorithm for $\pi(x)$ would probably be faster until roughly $x \approx 4 \cdot 10^{31}$. This is due to overhead, some of which is from the need of multiple precision complex arithmetic \cite{PL}.
   It seems likely the analytic algorithm for $M(x)$ would follow suit.
   
   \subsection{Analytic}
   It has been shown $\liminf q(x) < -1.837625$ and $\limsup q(x) > 1.826054$.
   Extending these bounds further, with the same approach, would take a considerable amount of time.
   To see why, first notice all values found with fplll, using $(\delta, \eta) = (0.9999, 0.99985)$, resulted in bounds about $95.5\%$ of the optimum for a given $N$, i.e.
   $$ h \approx 1.91\sum_{i =1}^N a_i. $$
   Additionally, the runtime of fplll's algorithm scales as $O(N^{4+\varepsilon} \nu (N + \nu))$ \cite{FPLLL}.
   Thus given the timings of previous calls and assuming $\nu$ scales linearly with $N$, these observations can help estimate what is needed to reach a given bound:
   \begin{center}
   \begin{tabular}{ | c || c | c | }
   \hline
   bound & estimated $N$ & estimated time\\
   \hline
   $1.90$ & $865$ & $\phantom{1}2$ months\\
   $1.95$ & $985$ & $\phantom{1}5$ months\\
   $2.00$ & $1125$ & $10$ months\\
   \hline
   \end{tabular}
   \end{center}
   
   It therefore appears attaining bounds of $\pm2$ is within reach with existing hardware and algorithms.
   Attaining bounds larger than $2$ will most likely need $\rho_i$ and $\zeta'(\rho_i)$ computed to higher precision than what was achieved here, or different $(\delta, \eta)$ values.
   At present, a different approach is likely needed to substantially improve these bounds past $2$.
   
   \section{Appendix}
   \noindent Access all computed data in a Mathematica notebook at \url{https://wolfr.am/mertens}.
   
   \section{Acknowledgements}
   The author wishes to thank the Texas Advanced Computing Center for providing the computing power to compute $\rho_i$ and $\zeta'(\rho_i)$ to such high precision and Michael Trott for referring me to the TACC. Additionally, a thanks goes out to Matthew Gelber and Eric Rowland for offering suggestions and edits throughout the writing process of this paper. Lastly, the author wishes to acknowledge Daniel Fortunato for his collaborations during the inception of this project.
   
% Bibliography.

\end{document}